\documentclass[10pt,reqno]{amsart}
\usepackage{ucs}

\usepackage{amsmath, amssymb, amscd}
\usepackage{pb-diagram}

\usepackage[latin1]{inputenc}
\usepackage{amsmath}
\usepackage{amssymb}
\usepackage{epsfig}
\usepackage{graphicx}
\usepackage{psfrag}
\usepackage{xypic}

\newtheorem{theorem}{Theorem}[section]
\newtheorem{definition}[theorem]{Definition}

\newtheorem{proposition}[theorem]{Proposition}
\newtheorem{corollary}[theorem]{Corollary}

\newtheorem*{remark}{Remark}

\renewcommand{\epsilon}{\varepsilon}

\DeclareMathAlphabet{\mathpzc}{OT1}{pzc}{m}{it}
\usepackage{mathrsfs}

\newcommand{\ol}{\overline}

\newcommand{\Z}{\mathbb{Z}}
\newcommand{\C}{\mathbb{C}}

\renewcommand{\qed}{$\hfill \square$ \smallskip \\}
\renewcommand{\phi}{\varphi}

\begin{document}
\thispagestyle{empty}
\title[The Kanenobu knots and Khovanov-Rozansky homology]{The Kanenobu knots and Khovanov-Rozansky homology}
\author{Andrew Lobb}

\email{lobb@math.sunysb.edu}
\address{Mathematics Department \\ Stony Brook University \\ Stony Brook NY 11794 \\ USA}

\begin{abstract}
Kanenobu has given infinite families of knots with the same HOMFLY polynomials.  We show that these knots also have the same $sl(n)$ and HOMFLY homologies, thus giving the first example of an infinite family of knots undistinguishable by these invariants.  This is a consequence of a structure theorem about the homologies of knots obtained by twisting up the ribbon of a ribbon knot with one ribbon.
\end{abstract}

\maketitle

\section{Context and results}
This paper has three sections.  In the first section we shall give our results and some context for them, postponing proofs for the second section.  In the final section we shall indicate a way to generalize our result.

\subsection{The Kanenobu knots}

\begin{figure}
\centerline{
{
\psfrag{p}{$p$}
\psfrag{q}{$q$}
\psfrag{-}{$-$}
\psfrag{ldots}{$\ldots$}
\psfrag{T(D)}{$T(D)$}
\psfrag{T-(D)}{$T^-(D)$}
\psfrag{T+(D)}{$T^+(D)$}
\includegraphics[height=2in,width=3.5in]{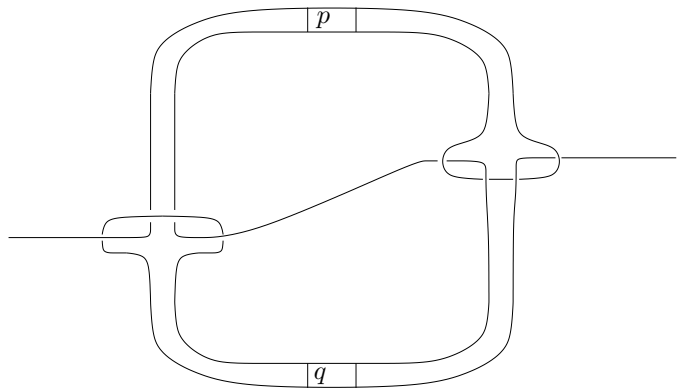}
}}
\caption{We give a diagram of the Kanenobu knots $K(p,q)$, passing through the point at infinity.  In this picture $|p|$ and $|q|$ are integers representing the number of half-twists added in the two strands, the sign of the half-twists depending on the signs of $p$ and~$q$.}
\label{knotskpq}
\end{figure}

Kanenobu has described knots $K(p,q)$ given by a pair of integers $p,q$ \cite{Kan}.  We draw these knots in Figure \ref{knotskpq}.  Kanenobu showed that these knots satisfy certain properties:

\begin{theorem}[Kanenobu \cite{Kan}]
\label{kanorig}
Suppose $p$ and $q$ are even.  Then we have
\begin{enumerate}
\item $K(p,q) = K(r,s) \Longleftrightarrow \{p,q\} = \{r,s\} \subset \Z$,
\item $\overline{K(p,q)} = K(-p,-q)$, and
\item $P(K(p,q)) = P(K(r,s))$ whenever $p+q = r+s$,
\end{enumerate}

\noindent where we write $\overline{K}$ for the mirror image of $K$, and $P(K)$ for the HOMFLY polynomial of $K$.  From Kanenobu's arguments, although he did not state it explicitly, we also have (2) and the $\Longleftarrow$ direction of (1) with no parity restriction on $p$ and $q$.
\end{theorem}

These families of knots $K(p,q)$ with $p+q$ constant and $p,q$ even are therefore infinite families of knots with the same HOMFLY and hence the same specializations of HOMFLY (the same Alexander polynomials, Jones polynomials, $sl(n)$ polynomials \emph{et cetera}).

Given a knot, Khovanov and Rozansky have defined bigraded vector spaces which recover the $sl(n)$ polynomials of the knot as the graded Euler characteristic \cite{KR1} and also trigraded vector spaces which recover the HOMFLY polynomial as the bigraded Euler characteristic \cite{KR2}.  Jacob Rasmussen \cite{Ras2} has given spectral sequences which start from the HOMFLY homology of a knot and converge to the $sl(n)$ homology.  As a consequence of the existence of these spectral sequences, one can think of the HOMFLY homology as being the limit as $n \rightarrow \infty$ of the $sl(n)$ homologies.

It is natural to ask whether the HOMFLY homology or, more generally, the $sl(n)$ homologies can detect the difference between $K(p,q)$ and $K(r,s)$ when $p+q = r+s$ and $p,q,r,s$ are all even.  In this paper we show that they cannot detect this difference.  As a consequence, the Kanenobu knots provide the first examples of an infinite collection of knots with the same HOMFLY and $sl(n)$ homologies.

\begin{theorem}
\label{kanethm}
We write $H_n(K)$ for the $sl(n)$ homology of a knot $K$.  We may mean the unreduced, reduced, or the equivariant homology with potential $x^{n+1} - (n+1)ax$.  Then for the Kanenobu knots $K(p,q)$ we have

\[ {H}_n(K(p,q)) = {H}_n(K(r,s)) \, \, {\rm whenever} \, \, p+q = r+s \, \, {\rm and} \, \, pq \equiv rs \pmod{2}{\rm .} \]
\end{theorem}

\begin{corollary}
\label{HHHkane}
We write $\ol{H}(K)$ for the reduced HOMFLY homology of a knot $K$.  Then for the Kanenobu knots $K(p,q)$ we have

\[ \ol{H}(K(p,q)) = \ol{H}(K(r,s)) \, \, {\rm whenever} \, \, p+q = r+s \, \, {\rm and} \, \, pq \equiv rs \pmod{2}{\rm .} \]
\end{corollary}

Liam Watson \cite{Wat} has an analogue of Theorem \ref{kanethm} for standard Khovanov homology over $\Z$, and Greene-Watson are working on analogues for odd Khovanov homology and for Heegaard-Floer knot homology.

\subsection{Ribbon knots with one ribbon}

The important point in our proof of Theorem \ref{kanethm} is that the knots $K(p,q)$ are ribbon knots with one ribbon.  If $K=K_0$ is a ribbon knot with one ribbon , then by twisting along the ribbon we obtain a sequence of knots $K_p$ for $p \in \Z$; this process is illustrated in Figure \ref{ribbon}.  For these knots we have the following result.

\begin{figure}
\centerline{
{
\psfrag{p}{$p$}
\psfrag{q}{$q$}
\psfrag{-}{$-$}
\psfrag{ldots}{$\ldots$}
\psfrag{T(D)}{$T(D)$}
\psfrag{T-(D)}{$T^-(D)$}
\psfrag{T+(D)}{$T^+(D)$}
\includegraphics[height=2in,width=3.5in]{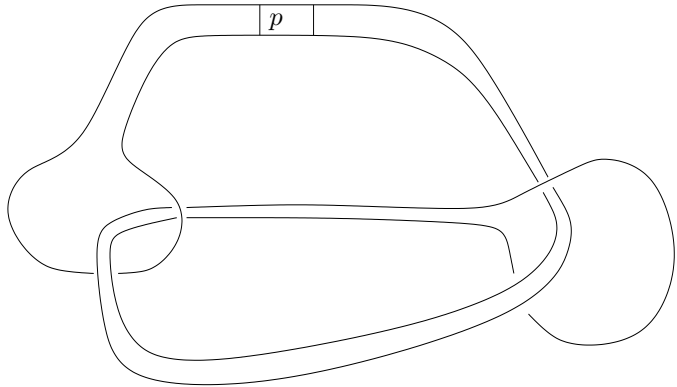}
}}
\caption{Here is an example of a class of ribbon knots $K_p$ with one ribbon.  On the ribbon we have inserted $|p|$ half-twists, positively or negatively depending on the sign on $p$.}
\label{ribbon}
\end{figure}

\begin{theorem}
\label{ribbonthm}
We write $H_n(K)$ for the $sl(n)$ homology of a knot $K$.  We may mean the unreduced, reduced, or the equivariant homology with potential $x^{n+1} - (n+1)ax$.  When working with equivariant homology every module in this statement should be read as a finitely-generated bigraded $\C[a]$-module, and otherwise as a finite dimensional bigraded $\C$-vector space.  Let $U$ be the unknot.  Then there exists a module $M_n(K_p)$ such that

\[{H}_n(K_p) = H_n(U) \oplus M_n(K_p) \]

\noindent and

\[ M_n(K_{p+2}) = M_n(K_p) [2] \{-2n\} \, \, {\rm for} \, \, {\rm all} \, \, p{\rm ,} \]

\noindent where the square brackets indicate a shift in the homological grading and the curly brackets indicate a shift in the quantum grading.
\end{theorem}

\begin{corollary}
\label{HHHribbon}
We write $\ol{H}^{i,j,k}(K)$ for the reduced HOMFLY homology of a knot $K$, using the grading conventions of \cite{Ras2}.  We use square brackets to denote a shift in the $j$-grading, and curly brackets to denote a shift in the $k$-grading.  Let $U$ be the unknot.  Then there exists a trigraded $\C$-vector space $M(K_p)$ such that

\[\ol{H}(K_p) = \ol{H}(U) \oplus M(K_p) \]

\noindent and

\[ M(K_{p+2}) = M(K_p) [-2] \{2\} \, \, {\rm for} \, \, {\rm all} \, \, p{\rm .} \]

%% \[ \ol{H}^{i,j,k}(K_{p+1}) = \ol{H}^{i,j-2,k+2} \]
\end{corollary}

\begin{remark}
In \cite{L2}, we indicated how the class of objects with well-defined Khovanov-Rozansky homologies could be enlarged to include knots with \emph{infinite twist sites}, which are sites where we add an infinite number of twists to two oppositely-oriented strands.  In the case of the ribbon knots $K_p$ considered in this subsection, Theorem \ref{ribbonthm} and Corollary \ref{HHHribbon} imply that the Khovanov-Rozansky homologies of $K_\infty$ are the same as those of the unknot.
\end{remark}

\subsection{Acknowledgements}
We thank Liam Watson for suggesting this problem and we thank Josh Greene for suggesting this problem later when we were better able to answer it.

\section{Proofs}
We organize this section somewhat in reverse, ending with our proof of Theorem~\ref{ribbonthm}.

\begin{proof}[Proof of Corollaries \ref{HHHkane} and \ref{HHHribbon}]
These Corollaries follow immediately from Theorems \ref{kanethm} and \ref{ribbonthm} and Rasmussen's Theorem $1$ from \cite{Ras2} which realizes the reduced HOMFLY homology as the limit of the reduced $sl(n)$ homologies as $n \rightarrow \infty$.
\end{proof}

\begin{proof}[Proof of Theorem \ref{kanethm} assuming Theorem \ref{ribbonthm}]
There are two places in Figure \ref{knotskpq} where adding a ribbon to the knot $K(p,q)$ will result in the $2$-component unlink: we could add a ribbon at one side of the $p$-twist region or at one side of the $q$-twist region.  Hence we observe that $K(p,q)$ is a ribbon knot with one ribbon in two possibly distinct ways.  This puts us in the situation of Theorem \ref{ribbonthm}.

Applying Theorem \ref{ribbonthm} at the two different sites tells us that

\[H_n(K(p,q)) = H_n(K(p+2, q-2)) \]

\noindent for all $p,q$.

Iterating this (and when $p+q$ is odd using the fact that $K(p,q) = K(q,p)$) we obtain the statement of Theorem \ref{kanethm}.
\end{proof}

Before we begin the proof of Theorem \ref{ribbonthm} we collect some results, many of which appeared in \cite{L2}.

\begin{figure}
\centerline{
{
\psfrag{K0}{$D_0$}
\psfrag{K-1}{$D_1$}
\psfrag{x1}{$x_1$}
\psfrag{x2}{$x_2$}
\psfrag{x3}{$x_3$}
\psfrag{x4}{$x_4$}
\psfrag{T+(D)}{$T^+(D)$}
\includegraphics[height=1.2in,width=5in]{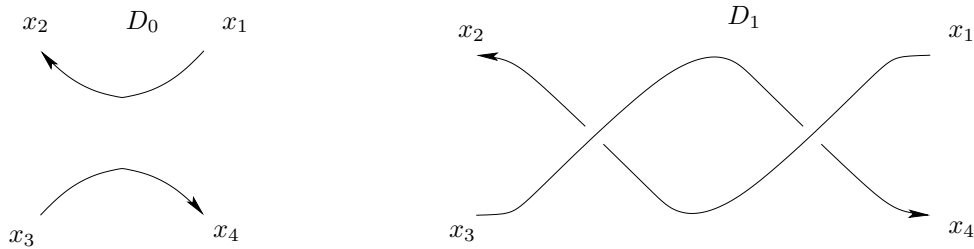}
}}
\caption{We have drawn two tangle diagrams, $D_0$ and $D_1$.  Associated in the $sl(n)$ Khovanov-Rozansky theory to each tangle diagrams is a complex of (vectors of) matrix factorizations.  We shall denote these up to chain homotopy equivalence by $C_n(D_0)$ and $C_n(D_1)$.}
\label{crossdiag}
\end{figure}

\begin{proposition}
\label{equivtrick}
Let $H_n(K)$ stand for the equivariant $sl(n)$ knot homology of $K$ with potential $x^{n+1} - (n+1)ax$ over the ring $\C[a]$.  Here $a$ is graded of degree $2n$ and $x$ of degree $2$ so that the potential is of homogeneous degree $2n+2$.

Then $H_n(K)$ has the structure of a $\C[a]$-module and the free part $F$ of $H_n(K)$ is of dimension $n$, supported entirely in homological degree $0$ and we have

\[F = H_n(U)\{s_n(K)\} {\rm ,}\]

\noindent where we write $U$ for the unknot, $s_n(K)$ for the analogue of the Lee-Rasmussen invariant $s(K)$ coming from Khovanov-Rozansky homology, and the curly brackets mean a shift in quantum degree.
\end{proposition}

\begin{proof}
Setting $a=1$ we recover Gornik's perturbation of Khovanov-Rozansky homology \cite{G}, which he showed was of dimension $n$ and supported in homological degree $0$.  This implies that the free part of $H_n(K)$ must be supported in homological degree $0$ and be of dimension $n$, since the free part is the only part that survives under setting $a=1$.  In \cite{L3} we have shown the dependence of the quantum degrees of Gornik's perturbation on a single even integer $s_n(K)$.  The result follows.
\end{proof}

Next consider the complexes $C_n(D_0)$ and $C_n(D_1)$ of matrix factorizations given in Figure \ref{crossdiag}.  Krasner has given a representative of the chain homotopy equivalence class of $C_n(D_1)$ which has a particularly simple form, the chain complex being made up of maps between matrix factorizations $V$ and $Z$ as in Figure \ref{VandZ}.

\begin{proposition}[Krasner \cite{Kras1}]
\label{basicprop}
Up to chain homotopy equivalence

\[C_n(D_1) = V[0]\{1-n\} \stackrel{x_2 - x_4}{\rightarrow} V[1]\{-1-n\} \stackrel{S}{\rightarrow} Z[2]\{-2n\} {\rm ,} \]

\noindent where square brackets indicate a shift in homological degree, curly brackets indicate a shift in quantum degree, and $S$ is the map induced by saddle cobordism.
\end{proposition}

We argued in \cite{L2} that this theorem holds even equivariantly, a flavour of Khovanov-Rozansky homology that did not exist when Krasner first formulated Proposition \ref{basicprop}.  Since the matrix factorization $Z$ is equal, up to degree shifts, to the only matrix factorization appearing in the complex $C(D_0)$, there is an obvious chain map induced by the identity map on $Z$:

\[ G : C_n(D_0) \rightarrow C_n(D_1)[-2]\{2n\} {\rm .} \]

\noindent We have shifted $C_n(D_1)$ here so that $G$ is graded of degree $0$ both in the homological and in the quantum gradings.

\begin{proposition}
\label{cone}
The cone of the chain map $G$, working either equivariantly or over $\C$, is the chain complex

\[Co(G) =  V[-2]\{1+n\} \stackrel{x_2 - x_4}{\rightarrow} V[-1]\{-1+n\} {\rm .} \]
\end{proposition}

\begin{proof}
This is a straightforward application of Gaussian elimination.
\end{proof}

The map $G$ appears in \cite{L2} as the map $G_{0,1}$.  In that paper, maps on knot homologies induced by maps such as $G$ were fitted together into large commutative diagrams with exact sequences for the rows.  For our application in the current paper we only need one of the rows, not the whole commutative diagram.

\begin{figure}
\centerline{
{
\psfrag{V}{$V$}
\psfrag{Z}{$Z$}
\psfrag{x1}{$x_1$}
\psfrag{x2}{$x_2$}
\psfrag{x3}{$x_3$}
\psfrag{x4}{$x_4$}
\psfrag{T+(D)}{$T^+(D)$}
\includegraphics[height=1.5in,width=3.5in]{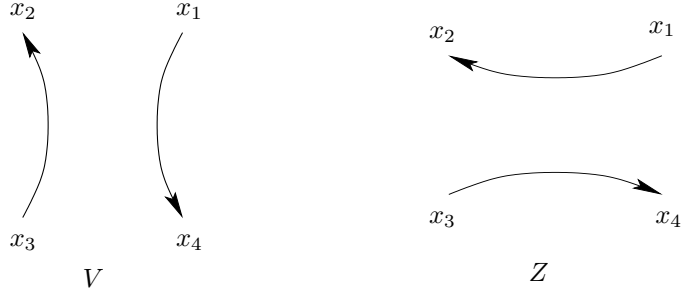}
}}
\caption{We have drawn here the matrix factorizations $V$ and $Z$.}
%In the text of this paper, $V$ and $Z$ often appear with integers appended in curly and/or square parentheses to indicate quantum degree shift and homological degree respectively.}
\label{VandZ}
\end{figure}

\begin{proposition}
\label{les}
Consider the knots $K_p$ as in the statement of Theorem \ref{ribbonthm}.  Then there is a long exact sequence in homology

\begin{eqnarray*}
\cdots &\rightarrow& N^{-1} \rightarrow H_n^0(K_p) \rightarrow H_n^2(K_{p+2})\{2n\} \rightarrow \\
&\rightarrow& N^0 \rightarrow  H_n^1(K_p) \rightarrow H_n^3(K_{p+2})\{2n\} \rightarrow \cdots
\end{eqnarray*}

\noindent where $N$ is a bigraded module given by

\[ N = H_n(U)[-1]\{0\} \oplus H_n(U)[-2]\{2n\} {\rm ,}\]

\noindent and we use superscripts to denote homological degree and all maps are of quantum-degree $0$.

This long exact sequence exists for equivariant, unreduced, or reduced flavours of homology.  In the first case, all modules are $\C[a]$-modules, in the other cases all modules are $\C$-vector spaces.
\end{proposition}

We note that this long exact sequence appears as the top line of the commutative diagram in Proposition $2.7$ of \cite{L2}.  

\begin{proof}
Since the knots $K_p$ and $K_{p+2}$ differ locally by replacing an occurrence of the tangle $D_0$ by the tangle $D_1$, the chain map $G$ induces a chain map

\[ G: C_n(K_p) \rightarrow C_n(K_{p+2})[-2]\{2n\} {\rm ,} \]

\noindent where we have written $C_n$ to denote the $sl(n)$ Khovanov-Rozansky chain complex.  If we let $\tilde{N}$ be the cone of this chain map then setting $N$ to be the homology of $\tilde{N}$ gives us the desired long exact sequence.  It remains to identify the module structure of $N$.

Note that Proposition \ref{cone} realizes $\tilde{N}$ as the cone of an explicit map between two chain complexes, each associated (up to some degree shifts) to a diagram of the $2$-component unlink.

Taking account of these degree shifts we see that $N$ is supported in homological degrees $-2$ and $-1$ and sits in a long exact sequence whose support is

\[ 0 \rightarrow N^{-2} \rightarrow H_n(U \cup U)\{1+n\} \stackrel{x-y}{\rightarrow} H_n(U \cup U)\{-1+n\} \rightarrow N^{-1} \rightarrow 0 {\rm .} \]

\noindent Here we are writing the homology of the $2$-component unlink as

\[ H_n(U \cup U) = \C[x,y]/(x^n = y^n = 0) \{2-2n\} \]

\noindent in the standard case and as

\[ H_n(U \cup U) = \C[a,x,y]/(x^n = y^n = a) \{2-2n\} \]

\noindent in the equivariant case.

Computing the kernel and cokernel of the map $x-y$ determines $N$ as in the statement of the Proposition.
\end{proof}

We are now ready to prove Theorem \ref{ribbonthm}.

\begin{proof}[Proof of Theorem \ref{ribbonthm}]
Proposition \ref{les} gives a long exact sequence relating the $sl(n)$ homologies of $K_p$ and $K_{p+2}$ with the bigraded module $N$, which is supported in only homological degrees $-2$ and $-1$ and whose structure has been explicitly computed.  If we can compute the maps in the long exact sequence

\[\phi: H_n^0(K_{p+2})\{2n\} \rightarrow N^{-2} \]

\noindent and

\[ \psi: N^{-1} \rightarrow H_n^0(K_p) \]

\noindent then we shall be able to describe $H_n(K_{p+2})$ completely in terms of $H_n(K_p)$.

Let us start by considering the equivariant case.  Since $K_p$ and $K_{p+2}$ are both smoothly slice, we have $s_n(K_p) = s_n(K_{p+2}) = 0$.  So Proposition \ref{equivtrick} tells us everything about the free parts of $H_n(K_p)$ and $H_n(K_{p+2})$.

The module $N^{-2}$ is a free $\C[a]$-module.  If the map $\phi$ is not an injection on the free part of $H_n^0(K_{p+2})\{2n\}$ then there must be some non-zero free part of $H^{-1}(K_p)$, but this cannot happen by Proposition \ref{equivtrick}.  Since $\phi$ preserves the quantum grading, we see that $\phi$ is therefore an isomorphism on the free part of $H_n^0(K_{p+2})\{2n\}$.  Similarly, we see that $\psi$ must map $N^{-1}$ isomorphically onto the free part of $H_n^0(K_p)$.

This establishes Theorem \ref{ribbonthm} in the equivariant case.

Specializing the equivariant case to $a=0$ we obtain the unreduced standard $sl(n)$ homology.  It is clear that $\phi$ descends to a surjective map when we specialize and that $\psi$ descends to an injective map.  This establishes the unreduced case.  Finally the reduced case follows from the unreduced case and the generalized universal coefficient theorem for principal ideal domains.
\end{proof}

\section{Extension of results}

\begin{figure}
\centerline{
{
\psfrag{V}{$V$}
\psfrag{Z}{$Z$}
\psfrag{x1}{$x_1$}
\psfrag{x2}{$x_2$}
\psfrag{x3}{$x_3$}
\psfrag{x4}{$x_4$}
\psfrag{T+(D)}{$T^+(D)$}
\includegraphics[height=1.5in,width=1.5in]{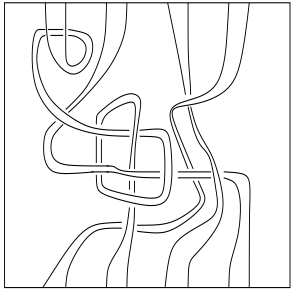}
}}
\caption{Here we show an example of a \emph{ribbon tangle} on $4$ ribbons.}
\label{ribbontanglething}
\end{figure}

The salient point in our proof of Theorem \ref{kanethm} was that the Kanenobu knots are ribbon knots on one ribbon in two different ways.  There are many ways to generate families of knots in which one can hope to retain this property.  As an example, we present one such way.

\begin{definition}
In Figure \ref{ribbontanglething} we draw a \emph{ribbon tangle} on $4$ ribbons.  A ribbon tangle on $n$ ribbons is a tangle with $n$ inputs at the bottom and $n$ outputs at the top, where we replace each tangle strand by two strands using the blackboard framing.
\end{definition}

We draw a knot $K_T(p_1,p_2,\ldots,p_n)$ in Figure \ref{kanenobugeneralized} depending on $n$ integers $p_i \in \Z$, and a ribbon tangle $T$.  By adding a ribbon to any of the twist regions we obtain the $2$-component unlink.  Hence we can apply Theorem \ref{ribbonthm} in this situation and so obtain the following result.

\begin{theorem}
Let $H$ be either reduced HOMFLY homology or $sl(n)$ homology (reduced, unreduced, or equivariant with potential $x^{n+1} - (n+1)ax$) and $T$ be any ribbon tangle.  Then we have that

\[ H(K_T(p_1,p_2, \ldots ,p_n)) = H(K_T(q_1,q_2,\ldots ,q_n))\]

\noindent whenever

\[ p_1 + p_2 + \cdots + p_n = q_1 + q_2 + \cdots + q_n \, \, {\rm and} \, \, p_i \equiv q_i \pmod{2} \, \, {\rm for} \, \, {\rm all} \, \, i {\rm .}\]
\qed
\end{theorem}

Of course, to show that in such examples you are generating infinitely many \emph{distinct} knots with the same Khovanov-Rozansky homologies, you need another invariant with which to distinguish between them.

\begin{figure}
\centerline{
{
\psfrag{p1}{$p_1$}
\psfrag{p2}{$p_2$}
\psfrag{p3}{$p_3$}
\psfrag{pn}{$p_n$}
\psfrag{T}{$T$}
\psfrag{ldots}{$\ldots$}
\psfrag{T+(D)}{$T^+(D)$}
\includegraphics[height=3in,width=4.5in]{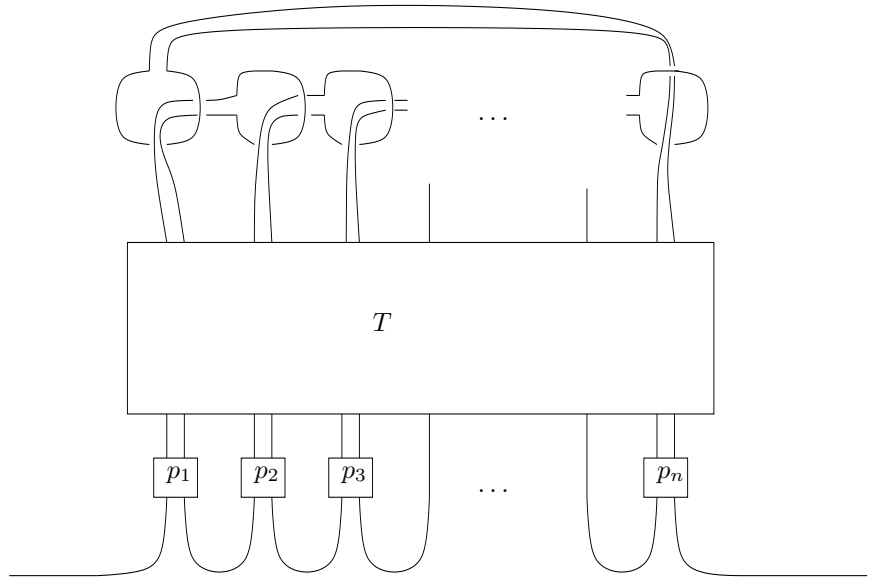}
}}
\caption{We have drawn a knot $K_T(p_1,p_2,\ldots,p_n)$ passing through the point $\infty$.  The knot depends on a ribbon tangle $T$ and $n$ integer parameters $p_1, p_2, \ldots, p_n$ describing the half-twists added to each ribbon.}
\label{kanenobugeneralized}
\end{figure}


\begin{thebibliography}{99999}

\bibitem{G} B.~Gornik, Note on Khovanov link cohomology, 2004, arXiv:math.QA/0402266

\bibitem{Kan} T.~Kanenobu, Infinitely many knots with the same polynomial invariant, \emph{Proc. Amer. Math. Soc.} 97 (1986) no. 1, 158--162.

\bibitem{KR1} M. ~Khovanov and L. ~Rozansky, Matrix factorizations and link homology I, \emph{Fund. Math.} 199 (2008), 1--91.

\bibitem{KR2} M. ~Khovanov and L. ~Rozansky, Matrix factorizations and link homology II, \emph{Geom. Topol.} 12 (2008), 1387--1425.

\bibitem{Kras1} D. ~Krasner, A computation in Khovanov-Rozansky homology, \emph{Fund. Math.} 203 (2009), 75--95.

\bibitem{L2} A. ~Lobb, 2-strand stabilization and knots with identical quantum knot homologies, arXiv:1103.1412.

\bibitem{L3}A.~Lobb, A note on Gornik's perturbation of Khovanov-Rozansky homology, arXiv:1012.2802.

\bibitem{Ras2} J.~Rasmussen, Some differentials on Khovanov-Rozansky homology, arXiv:math/0607544v2.

\bibitem{Wat} L.~Watson, Knots with identical {K}hovanov homology, \emph{Algebr. Geom. Topol.} 7 (2007), 1389--1407.

\end{thebibliography}
\end{document}